\documentclass[12pt,reqno]{amsart}
\usepackage{amsmath,amsthm,amssymb,amsfonts,amscd}
\usepackage{mathrsfs}
\usepackage{bbm}
\usepackage{bbding}
\usepackage{hyperref}
\usepackage{geometry}
\usepackage{color}
\usepackage{xcolor}

\usepackage{picture,epic}
\usepackage{tikz}

\numberwithin{equation}{section}

\theoremstyle{plain}
\newtheorem{theorem}{Theorem}
\newtheorem{lemma}[theorem]{Lemma}
\newtheorem{corollary}[theorem]{Corollary}

\theoremstyle{definition}

\theoremstyle{remark}
\newtheorem{remark}{Remark}

\renewcommand{\Re}{\operatorname{Re}}
\renewcommand{\Im}{\operatorname{Im}}
\newcommand{\vol}{\operatorname{vol}}

\newcommand{\supp}{\operatorname{supp}}
\newcommand{\sym}{\operatorname{sym}}

\newcommand{\N}{\operatorname{N}}

\newcommand{\SL}{\operatorname{SL}}
\newcommand{\PSL}{\operatorname{PSL}}
\renewcommand{\mod}{\operatorname{mod}\ }

\newcommand{\dd}{\mathrm{d}}

\makeatletter
\def\@tocline#1#2#3#4#5#6#7{\relax
  \ifnum #1>\c@tocdepth 
  \else
    \par \addpenalty\@secpenalty\addvspace{#2}%
    \begingroup \hyphenpenalty\@M
    \@ifempty{#4}{%
      \@tempdima\csname r@tocindent\number#1\endcsname\relax
    }{%
      \@tempdima#4\relax
    }%
    \parindent\z@ \leftskip#3\relax \advance\leftskip\@tempdima\relax
    \rightskip\@pnumwidth plus4em \parfillskip-\@pnumwidth
    #5\leavevmode\hskip-\@tempdima
      \ifcase #1
       \or\or \hskip 1em \or \hskip 2em \else \hskip 3em \fi%
      #6\nobreak\relax
    \hfill\hbox to\@pnumwidth{\@tocpagenum{#7}}\par
    \nobreak
    \endgroup
  \fi}
\makeatother

\begin{document}

\title[Sup-norm and nodal domains of dihedral Maass forms]
{Sup-norm and nodal domains of dihedral Maass forms}
\author{Bingrong Huang}
\address{School of Mathematical Sciences \\ Tel Aviv University \\ Tel Aviv \\ Israel}
\email{bingronghuangsdu@gmail.com}
\date{\today}

\begin{abstract}
  In this paper, we improve the sup-norm bound and the lower bound of the number of nodal domains for dihedral Maass forms, which are  a distinguished sequence of Laplacian eigenfunctions on an arithmetic hyperbolic surface.
  More specifically,
  let $\phi$ be a dihedral Maass form with spectral parameter $t_\phi$, then we prove that $\|\phi\|_\infty \ll t_\phi^{3/8+\varepsilon} \|\phi\|_2$, 
  which is an improvement over the bound $t_\phi^{5/12+\varepsilon} \|\phi\|_2$ given by Iwaniec and Sarnak.
  As a consequence, 
  we get a better lower bound for the number of nodal domains intersecting a fixed geodesic segment under the Lindel\"{o}f Hypothesis.
  Unconditionally, we prove that the number of nodal domains grows faster than $t_\phi^{1/8-\varepsilon}$ for any $\varepsilon>0$ for almost all dihedral Maass forms.
\end{abstract}


\thanks{The work was supported  by the European Research Council, under the European Union's Seventh
Framework Programme (FP7/2007-2013)/ERC grant agreement n$^{\text{o}}$~320755.}

\maketitle

\section{Introduction} \label{sec:Intr}

\subsection{Sup-norm}

The distribution of mass of an eigenfunction of
the Laplacian on a Riemannian surface $X$
has received a lot of attention in the context of quantum chaos.  
Berry \cite{berry1977rugular} suggested that eigenfunctions for chaotic systems are modeled by random waves.
In particular, one would like to compare the sup-norm $\|\phi\|_\infty$ of an $L^2$-normalized eigenfunction $\phi$ with the corresponding quantity for random waves, which grows very slowly, as $\sqrt{\log \lambda_\phi}$, if $\lambda_\phi$ is the corresponding eigenvalue (see Salem--Zygmund \cite[Ch. IV]{SalemZygmund1954properties}).
The bound $\|\phi\|_\infty \ll \lambda_\phi^{1/4}$ is valid on any compact Riemannian surface (see Seeger--Sogge \cite{SeegerSogge1989bounds}), which is sharp for standard $2$-sphere.
However, this bound is not optimal for most surfaces.
Especially, 
Sarnak \cite{sarnak1995arithmetical} conjectured that, for compact surfaces of negative curvature, $\|\phi\|_\infty \ll \lambda_\phi^\varepsilon$ for all $\varepsilon>0$.

A related issue is whether such eigenfunctions have quantum uniquely ergodicity (QUE), that is for the sequence of eigenfunctions $\{\phi_j\}_{j\geq1}$, if we have
\[
  \lim_{j\rightarrow\infty} \langle A \phi_j ,\phi_j \rangle_X = \int_{S^* X} \sigma_A
\]
for all 0th order pseudodifferential operators $A$ with principal symbol $\sigma_A$.
Here $S^* X$ is the unit cotangent bundle.
See \cite{sarnak1995arithmetical} and \cite{RudnickSarnak1994behaviour} for more details.

The first breakthrough of the sup-norm problem 
was achieved by Iwaniec--Sarnak \cite{IwaniecSarnak1995norms}, who proved that for certain arithmetic (compact) hyperbolic surfaces and for Hecke eigenforms $\phi$, we have
\begin{equation}\label{eqn:IS}
  \|\phi\|_\infty \ll_\varepsilon \lambda_\phi^{5/24+\varepsilon}.
\end{equation}
This is to be compared with the ``local Weyl law'', which gives $\lambda_\phi^{1/4} /\log \lambda_\phi$ for any negative curvature surface (see \cite{berard1977wave}).
Iwaniec--Sarnak \cite{IwaniecSarnak1995norms} also proved the same bounds \eqref{eqn:IS} for the modular surface, which is non-compact but of finite volume.
To be specific, we consider the 
arithmetic surface $\mathbb{X}=\Gamma_0(q)\backslash\mathbb{H}$, where
$$\Gamma_0(q)=\{\left(\begin{smallmatrix}
                   a & b \\
                   c & d
                 \end{smallmatrix}\right)\in\SL_2(\mathbb{Z}): c\equiv0 \ (\mod q)\}
$$
is the Hecke congruence subgroup,
and $\mathbb H$ is the Poincar\'{e} upper half-plane equipped with its hyperbolic metric $\dd s= \sqrt{\dd x^2 + \dd y^2}/y$ and associated measure $\dd x\dd y/y^2$. Then \eqref{eqn:IS} holds
for a weight zero Hecke--Maass cusp form $\phi$ of level 1 (i.e. $q=1$).
For experiments of arithmetical quantum chaos in this context, see \cite{HejhalRackner1992topography}.
See also
\cite{BlomerHolowinsky2010bounding,templier2010sup,
HarcosTemplier2012sup,HarcosTemplier2013sup,templier2015hybrid,saha2017hybrid} 
for sup-norm bounds in other aspects of Hecke--Maass forms.

\medskip

In this paper we deal with a distinguished sequence of eigenfunctions, called dihedral Maass forms. These give a sparse subsequence of all eigenfunctions, for which we can improve on the result of Iwaniec and Sarnak.
The number of such eigenfunctions with eigenvalue up to $X$ grows like $c\sqrt{X}$ for some constant $c$, compared to the count for all eigenfunctions, which by Weyl's law is about $c'X$ for some constant $c'$.
See \cite{HejhalStromergsson2001quantum} for experiments of arithmetical quantum chaos of dihedral Maass forms.
In more detail, let $K=\mathbb{Q}(\sqrt{q})$ be a fixed real quadratic field with discriminant $q>0$ squarefree and $q\equiv1$ (mod 4).
For simplicity, we assume that $K$ has the narrow class
number 1 and $q>8$ is a prime.
It is conjectured that there are infinitely many such $q$'s. For example, we may take
$q = 13, 17, 29, 37, 41, 53,$ and so on.
Let $\omega_q=\frac{1+\sqrt{q}}{2}$ and let $\epsilon_q$ be the fundamental unit of $K$.
The ring of integers of $K$ is $\mathcal{O}_K=\mathbb{Z}[\omega_q]$,
and the group of units $U_K$ in $\mathcal{O}_K$ is isomorphic to $\{\pm1\}\times \epsilon_q^{\mathbb{Z}}$.
For integer $k\neq0$, we have the
Hecke Gr\"{o}ssencharacter $\Xi_k$ of $K$ defined by
\[
  \Xi_k((\alpha)) := \Big|\frac{\alpha}{\tilde{\alpha}}\Big|^{\frac{\pi i k}{\log \epsilon_q}} \quad
  \textrm{for ideal $(\alpha)\subset \mathcal O_K$ with generator $\alpha$,}
\]
where $\tilde{\alpha}$ is the conjugate of $\alpha$ under the nontrivial automorphism of $K$.
By \cite{maass1949uber}, we know that the theta-like series associated to $\Xi_k$ by
\[
  \phi_k(z) := y^{1/2} \sum_{\beta\in\mathcal{O}_K/\pm U^+}
  \Xi_k((\beta)) K_{\frac{ik\pi}{\log \epsilon_q}}(2\pi |\N(\beta)|y) e(\N(\beta)x)
\]
is a Hecke--Maass cusp form on $\Gamma_0(q)$ of weight 0,
eigenvalue $\frac{1}{4}+(\frac{k\pi}{\log \epsilon_q})^2$,
and with nebentypus character $\chi_q$ (the Kronecker symbol).
Here $U^+$ denotes the group of totally positive units of $\mathcal{O}_K$,
$\N(\beta)=\beta\tilde{\beta}$ is the norm of $\beta$, $K_\nu(z)$ is the modified Bessel function, and $e(x)=e^{2\pi ix}$.
Since $\N(\epsilon_q)=-1$, we know $\phi_k(z)$ is even.
Denote
\[ T_k := t_{\phi_k} = \frac{\pi k}{\log \epsilon_q}. \]
Note that $\phi_k(z)$ is real since we may write
$$
  \phi_k(z) = 2y^{1/2}\sum_{n=1}^{\infty} a_{k}(n) K_{iT_k}(2\pi ny) \cos(2\pi nx),
$$
where $a_{k}(n):=\sum_{\N((\beta))=n}\Xi_k((\beta))$ is real.
Note that $\phi_k$ is Hecke normalized, since $a_k(1)=1$.
By the Rankin--Selberg method, we obtain (see \cite{HoffsteinLockhart1994coefficients})
\begin{equation}\label{eqn:L^2}
  \|\phi_k\|_2^2 =  \int_{\Gamma_0(q)\backslash\mathbb{H}} |\phi_k(z)|^2 \frac{\dd x\dd y}{y^2} =  (q T_k)^{o(1)} e^{-\pi T_k}.
\end{equation}

\medskip
Our first main result is the following theorem.
\begin{theorem}\label{thm:sup-norm}
  With the notation as above, we have
  \[ \|\phi_k\|_\infty \ll_{q,\varepsilon} T_k^{3/8+\varepsilon}\|\phi_k\|_2 . \]
  That is, we have $\phi_k(z) \ll_{q,\varepsilon} T_k^{3/8+\varepsilon} \|\phi_k\|_2$, which is uniform in $z\in\mathbb{H}$.
\end{theorem}

Define $L^p$-norm of $\phi_k$ by ($p\geq1$)
\[
  \|\phi_k\|_p := \Big( \int_{\Gamma_0(q)\backslash\mathbb{H}} |\phi_k(z)|^p \frac{\dd x\dd y}{y^2} \Big)^{1/p}.
\]
In \cite{luo2014norms}, Luo proved a sharp upper bound for $L^4$-norm
\[
  \|\phi_k\|_4 \ll_{q,\varepsilon} T_k^\varepsilon \|\phi_k\|_2.
\]
Together with our sup-norm bound, we get the following result.

\begin{corollary}
  We have
  \[
    \frac{\|\phi_k\|_p}{\|\phi_k\|_2}
    \ll_{q,\varepsilon} \left\{\begin{array}{ll}
                               T_k^\varepsilon, & \textrm{if $2<p\leq 4$,} \\
                               T_k^{3(1-4/p)/8+\varepsilon}, & \textrm{if $p>4$.}
                             \end{array}\right.
  \]
\end{corollary}

It is instructive to compare our results for dihedral Maass forms with those for Eisenstein series.
The key point here (from the automorphic representation theoretic point of view) is that if $K$ is a quadratic etale algebra over $\mathbb Q$ and $\chi$ is a Hecke character of $K$, then the principle of automorphic induction gives us an automorphic representation of ${\rm GL}_2(\mathbb Q)$. If $K$ is split, i.e., $K = \mathbb Q \times \mathbb Q$, then this leads to an Eisenstein series. On the other hand if $K$ is a field, then we end up getting a cusp form; the sup-norm in the latter case is exactly what we consider in this paper (under additional restrictions), while the sup-norm in the former case has been considered by several others in recent years, starting with the paper of Young.
By using the method of Iwaniec--Sarnak, Young \cite{young2018note}
established that for $\Omega$ a fixed compact subset of $\mathbb H$, and $T\geq 1$,
\begin{equation*}
  \max_{z\in\Omega} |E(z,1/2+iT)| \ll_{\Omega,\varepsilon} T^{3/8+\varepsilon},
\end{equation*}
where $E(z,s)$ is the usual real-analytic Eisenstein series for the group $\PSL_2(\mathbb Z)$.
This was improved by Blomer \cite{blomer2016epstein} to
\[
  \max_{z\in\Omega} |E(z,1/2+iT)| \ll_{\Omega,\varepsilon} T^{1/3+\varepsilon},
\]
by using a method of Titchmarch.
%
Recently, Huang--Xu \cite{HuangXu2017sup-norm} proved
\[
  E(z,1/2+iT) = c_0(y,1/2+iT) + O(T^{3/8+\varepsilon}),
\]
if $y\gg1$, where $c_0(y,s)$ is the constant term in the Fourier expansion of $E(z,s)$; Assing \cite{Assing2017} extended to the number fields case.
To prove our results, we will follow the approach in  \cite{HuangXu2017sup-norm}.

\subsection{Nodal domains}

Let $\phi$ be a Hecke--Maass cusp form on $\mathbb{X}$.
Let $Z_\phi$ be the zero set of $\phi$,
which is a finite union of real analytic curves.
For any subset $C \subseteq \mathbb{X}$, let $N^C\left(\phi\right)$ be the number of connected components (the nodal domains) in $\mathbb{X} \backslash Z_\phi$ which intersect $C$. Let $N\left(\phi\right)=N^\mathbb{X}\left(\phi\right)$.

In \cite{BogomolnySchmit2002percolation}, Bogomolny--Schmit estimated the expected number of nodal domains of random waves using a percolation like model. In view of Berry's conjecture, 
results in \cite{BogomolnySchmit2002percolation} suggests the existence of a constant $c>0$ such that
\begin{equation}\label{eqn:nodal}
N\left(\phi\right) \sim c \lambda_\phi. 
\end{equation}
In \cite{ns}, Nazarov--Sodin examined \eqref{eqn:nodal} for random spherical harmonics, and they proved the existence of $c>0$ such that \eqref{eqn:nodal} holds almost surely as $\lambda_\phi \to +\infty$.
Note that it is not true for a general Riemannian surface that the number of nodal domains of an eigenfunction must increase with the eigenvalue \cite{st,lewy}.

Let $\sigma : \SL_2(\mathbb{Z})\backslash \mathbb{H} \to \SL_2(\mathbb{Z})\backslash \mathbb{H}$ be an orientation reversing isometric involution induced by $x+iy \mapsto -x+iy$ on $\mathbb{H}$.
Let the arc $\delta=\{z\in \SL_2(\mathbb{Z})\backslash \mathbb{H}:\sigma(z)=z\}$ which is fixed by $\sigma$. In \cite{GhoshReznikovSarnak2013nodal}, Ghosh--Reznikov--Sarnak studied nodal domains crossing $\delta$
and proved
\begin{equation}\label{capt}
  t_\phi \ll N^\delta\left(\phi\right) \ll t_\phi \log t_\phi
\end{equation}
for even Hecke--Maass cusp forms $\phi$. Here $t_\phi>0$ is the spectral parameter of $\phi$ with $\lambda_\phi=1/4+t_\phi^2$. Assuming that \eqref{eqn:nodal} is true, this estimate in particular implies that almost all nodal domains do not touch $\delta$.
In order to prove the lower bound in  \eqref{capt}, the authors produced sign changes of Hecke--Maass cusp forms high in the cusp, $t_\phi >y > t_\phi/100$.
For nodal domains intersecting a fixed sufficiently long geodesic segment $\beta \subset \{iy:y\geq1\}$ or $\beta\subset \{1/2+iy:y\geq\sqrt{3}/2\}$, they proved
\[
  N^\beta\left(\phi\right) \gg_\varepsilon t_\phi^{1/12-\varepsilon},
\]
by assuming the Lindel{\"o}f Hypothesis for the L-functions $L\left(s,\phi\right)$.
Recently in Jang--Jung \cite{JangJung2018quantum}, they showed $\lim_{t_\phi \to \infty} N^\beta\left(\phi\right) = +\infty$ without any assumptions. However no quantitative lower bound is given in \cite{JangJung2018quantum}.

\medskip
In this paper, we will consider the family of dihedral Maass forms $\phi_k$ with $K<k\leq 2K$. Let $L(s,\phi_k)$ be the Hecke L-function of $\phi_k$.
As a consequence of Theorem \ref{thm:sup-norm}, we have the following result.
\begin{corollary}
  Let $\beta \subset \left\{iy:y>0\right\}$ be any fixed compact geodesic segment.
  Assume the Lindel{\"o}f Hypothesis for the $L$-functions $L\left(s,\phi_k\right)$.  Then for any $\varepsilon>0$, we have
  \[
    |Z_{\phi_k} \cap \beta| \gg_{q,\varepsilon} T_k^{1/8-\varepsilon}
    \quad \textrm{and} \quad
    N^\beta(\phi_k) \gg_{q,\varepsilon} T_k^{1/8-\varepsilon} .
  \]
\end{corollary}

In \cite{jung2016quantitative}, Jung unconditionally proved a lower bound $t_\phi^{1/8-\varepsilon}$ for almost all $\phi$ within the set of even Hecke--Maass cusp forms in $\left\{\phi:|T-t_\phi|<T^{1/3}\right\}$ by quantitative quantum ergodicity and a sharp estimate on the variance of the shifted convolution sums.
%
%
In order to get unconditional results on nodal intersections and nodal domains for the family of dihedral Maass forms, we will need the following second moment of Hecke L-functions.
\begin{theorem}\label{thm:second_moment}
  Let $K\geq2$ be sufficiently large.
  Let 
  $|t|\ll K^{1-\varepsilon}$.
  Then, for any $\varepsilon>0$, we have
  \[
    \frac{1}{K} \sum_{K<k\leq 2K} |L(1/2+it,\phi_k)|^2 \ll_{q,\varepsilon} K^{\varepsilon}.
  \]
\end{theorem}

Now we 
state our unconditional result for dihedral Maass forms as an application of our estimates on the sup-norm bounds and the second moment of Hecke L-functions.
\begin{theorem}\label{thm:nodal_domain}
  Let $\beta \subset \left\{iy:y>0\right\}$ be any fixed compact geodesic segment.
  Then for any $\varepsilon>0$, all but $O_{q,\varepsilon}\left(K^{1-\varepsilon}\right)$ dihedral Maass forms with $K<k\leq 2K$ satisfy
  \[
    |Z_{\phi_k} \cap \beta| \gg_{q,\varepsilon} T_k^{1/8-\varepsilon}
    \quad \textrm{and} \quad
    N^\beta(\phi_k) \gg_{q,\varepsilon} T_k^{1/8-\varepsilon} .
  \]
\end{theorem}

\begin{remark}
  Jung's result \cite{jung2016quantitative} does not apply here, because our sequence of dihedral Maass forms is too sparse to be captured by his almost-everywhere results.
\end{remark}

\begin{remark}
  Recently, in \cite{JungYoung2016sign}, Jung--Young proved a quantitative lower bound on the number of nodal domains of the real-analytic Eisenstein series, by a quantitative restricted QUE theorem.
\end{remark}


\subsection{Key ideas}

We will use the amplification method to reduce our sup-norm problem to an interesting lattice point counting, as \cite{IwaniecSarnak1995norms} did.
The key observation in our improvement of the sup-norm for dihedral Maass forms is that we can obtain a nice lower bound (in fact, an asymptotic formula) for an amplifier which is supported only on the primes (Lemma \ref{lemma:amplifier}).
The main ingredients to prove the lower bound are the factorization of the symmetric square L-function of a dihedral form \eqref{eqn:sym^2} and a good zero-free region and an upper bound of the logarithmic derivative of an Hecke L-function (Lemmas \ref{lemma:zero-free-region} \& \ref{lemma:upper-bound-L'/L}).
A similar result can be proved for an amplifier supported on the integers as in \cite[Remark 1.6]{IwaniecSarnak1995norms}, which, however, can not get us a uniform upper bound with exponent $3/8$ for all $z\in\mathbb{H}$ (see \cite[\S6]{young2018note}).
By requiring our amplifier to be supported on the primes, we can use the geometric method as was done in \cite{HarcosTemplier2013sup,templier2015hybrid,blomer2016sup-norm} to obtain a more efficient treatment for the counting problem.

To prove Theorem \ref{thm:second_moment}, we will use a large sieve inequality for the family of Hecke Gr\"{o}ssen-characters, which turns out to be a consequence of Montgomery--Vaughan's large sieve inequality \cite{montgomery1973large}.
To prove this claim, we need an interesting elementary result on the repulsion of angles for a real quadratic field (Lemma \ref{lemma:angle-gap}), which shows that angles are well-spaced.
This repulsion between angles is also an important fact when one studies angle distribution in short arcs (see \cite{RudnickWaxman2017angles} for the imaginary quadratic field $\mathbb{Q}(i)$).

\medskip

The plan of this paper is as follows.
In \S\ref{sec:preliminaries}, we recall some results on Hecke--Maass forms, the amplified pre-trace formula, and counting lattice points, which will be used to prove the sup-norm bound. In \S\ref{sec:HeckeLfunction}, we discuss properties of an Hecke L-function. And then in \S\ref{sec:amplifier}, we define an amplifier which is only supported on the primes, and prove a lower bound for the amplifier. 
Combining those results, in \S\ref{sec:supnorm}, we complete the proof of Theorem \ref{thm:sup-norm}.
In \S\ref{sec:secondmoment}, we prove
the second moment of Hecke L-functions (Theorem \ref{thm:second_moment}). Finally, in \S \ref{sec:nodaldomains}, we follow the arguments in \cite{GhoshReznikovSarnak2013nodal} to prove Theorem \ref{thm:nodal_domain}. 

\medskip
Throughout the paper, $\varepsilon$ is an arbitrarily small positive number,
while $c$ and $c'$ 
stand for some absolute positive constants;
all of them may be different at each occurrence.

\subsection*{Acknowledgements}
The author would like to thank Prof. Ze\'{e}v Rudnick for suggesting thinking about dihedral forms, and for his valuable discussions and constant encouragement.
He also wants to thank Professors Gergely Harcos, Peter Humphries, Junehyuk Jung, Djordje Mili\'{c}evi\'{c}, and Matthew Young for their interest, comments, and suggestions.
The author gratefully thanks to the referees for the constructive comments and recommendations which definitely help to improve the readability and quality of the paper, and especially simplify the proof of Lemma \ref{lemma:AFE}.




\section{Preliminaries}\label{sec:preliminaries}

\subsection{Hecke--Maass forms}
Let $\mathbb{H}=\{z=x+iy:y>0\}$ be the upper half-plane,
and $\Gamma_0(q)$ the Hecke congruence group of level $q$.
The Laplace operator is $\Delta = -y^2\left(\frac{\partial^2}{\partial x^2} + \frac{\partial^2}{\partial y^2} \right).$

Throughout this paper we assume $q$ be the same as in the introduction and $\chi=\chi_q$ be the Kronecker symbol of $K=\mathbb{Q}(\sqrt{q})$.
Let $\phi$ be a cuspidal Hecke--Maass newform  of level $q$ and nebentypus character $\chi$. It satisfies the automorphy condition
\[
  \phi(\gamma z) = \chi(d) \phi(z), \quad
  \gamma=\left(\begin{smallmatrix}a&b\\c&d\end{smallmatrix}\right) \in\Gamma_0(q), \quad z\in\mathbb{H},
\]
and is an eigenfunction of the Laplace operator: $\Delta \phi= \lambda \phi$, and of the Hecke operators.
It is also an eigenfunction of the Atkin--Lehner operators.
We have the following upper bound of $\phi(z)$ by bounding its Fourier expansion.

\begin{lemma}\label{lemma:Bound_FE}
  With the notation as above, we have
  \[
    \phi(z) \ll_{q,\varepsilon} \lambda^{1/4+\varepsilon}/y^{1/2} + \lambda^{1/12+\varepsilon}.
  \]
\end{lemma}
\begin{proof}
  See e.g. \cite[\S3]{templier2015hybrid}.
\end{proof}

\subsection{The amplified pre-trace formula}
Let $k\in\mathcal{C}^\infty([0,\infty))$ with rapid decay.
Then, it can be viewed as the inverse of the Selberg transform
of a function $h(t)$, which is given by the following three steps
(see \cite[(1.64)]{iwaniec2002spectral}:
\begin{equation*}
  \begin{split}
    g(\xi)&:=\frac{1}{2\pi i}\int_{-\infty}^\infty e^{-ir\xi}h(r)\dd r, \\
    2q(v) &:=g(2\log(\sqrt{v+1}+\sqrt{v})),\\
    k(u)  &:=-\frac{1}{\pi}\int_{u/4}^{\infty}(v-u/4)^{-1/2}\dd q(v).
  \end{split}
\end{equation*}
Assume now that $k(z,w)=k(u(z,w))$ is a point pair
invariant kernel with
\begin{equation*}
  u(z,w):=\frac{|z-w|^2}{\Im (z)\Im(w)},\ \ z,w\in\mathbb{H},
\end{equation*}
and $h(t)$ is the corresponding Selberg transform which satisfies the conditions (see \cite[(1.63)]{iwaniec2002spectral})
\begin{equation*}\label{eqn:h(r)-conditions}
  \left\{\begin{array}{l}
    h(t)  \textrm{ is even,}  \\
    h(t)  \textrm{ is holomorphic in the strip } |\Im t | < \frac{1}{2} + \varepsilon, \\
    h(t)  \ll (|t|+1)^{-2-\varepsilon}  \textrm{ in the strip.}
  \end{array}\right.
\end{equation*}
For $\ell\geq1$, let
\[
  \mathcal M(\ell,q) := \big\{ \gamma=\left(\begin{smallmatrix}a&b\\c&d\end{smallmatrix}\right)\in M_2(\mathbb{Z}): \det(\gamma)=\ell, \ c\equiv0 \ (\mod q) \big\}.
\]
We embed our dihedral Maass form $\phi_k$ into an orthogonal basis $(u_j)_{j\geq0}$ of Hecke--Maass eigenforms of level $q$ and nebentypus $\chi$, with $u_0$ a constant function and $u_j$ cuspidal otherwise.
Using the pre-trace formula and Hecke operators, we obtain the {\it amplified pre-trace formula} (see \cite{IwaniecSarnak1995norms} and \cite[\S3]{HarcosTemplier2012sup})
\begin{equation}\label{eqn:APTF}
  \sum_{j\geq0}h(t_j)A_{u_j} \frac{|u_j(z)|^2}{\|u_j\|_2^2}
  + \mathrm{cont.}
  = \sum_{\ell} \frac{y_\ell}{\sqrt{\ell}} \sum_{\gamma\in \mathcal M(\ell,q)} k(\gamma z,z),
\end{equation}
where $h:\mathbb{R}\cup[-i/2,i/2]\rightarrow\mathbb{R}_+$ is a positive even smooth function of rapid decay, $(x_n)$ is sequence of complex numbers supported on finitely many $n$'s, $y_\ell$ is defined by
\begin{equation}\label{eqn:y_ell}
  y_\ell := \sum_{\substack{d|(m,n)\\ \ell=mn/d^2}} x_m\overline{x_n}
  = \sum_{\substack{d\geq 1\\ \ell=\ell_1\ell_2}} x_{d\ell_1}\overline{x_{d\ell_2}},
\end{equation}
and
\begin{equation}\label{eqn:A}
  A_{u_j} := \Big|\sum_{n}x_n\lambda_j(n)\Big|^2.
\end{equation}
Here ``cont.'' stands for an analogous positive contribution of the continuous spectrum.

\subsection{Counting lattice points}

Let $\mathcal{F}(q)$ be the set of $z\in\mathbb{H}$ such that
$\Im(z)\geq \Im(A z)$ for all Atkin--Lehner operators $A$ of level $q$. See \cite[\S2]{HarcosTemplier2012sup}.
By Lemma 2.2 there, we have $qy\gg1$ for all $z\in\mathcal{F}(q)$.
In fact, since $q$ is prime in our case, the Atkin--Lehner theory reduces to the Fricke involution $W_q=\frac{1}{\sqrt{q}} \left(\begin{smallmatrix}&-1\\q&\end{smallmatrix}\right)$.
Our Hecke--Maass newform $\phi_k$ is an eigenvector for $W_q$ with eigenvalue $\pm1$, therefore in order to bound its sup-norm, we can restrict ourselves to $\mathcal{F}(q)$. In particular, since the two cusps $0$ and $\infty$ are interchanged by $W_q$, we can assume
\[ y\gg_q 1. \]

For $z\in\mathcal F(q)$, $\delta>0$ and the integer $\ell$, let
\begin{equation*}
  \mathcal M(z,\ell,\delta,q) := \left\{ \gamma=\left(\begin{smallmatrix}a&b\\c&d\end{smallmatrix}\right) \in M_2(\mathbb Z):
  \det(\gamma)=\ell, \quad c\equiv0 \ (\mod q), \quad
  u(\gamma z,z)\leq \delta \right\}.
\end{equation*}
Denote by $M = M(z,\ell,\delta,q)$ its cardinality.
We split the counting $M$ of matrices $\gamma=\left(\begin{smallmatrix}a&b\\c&d\end{smallmatrix}\right)$ as
\[
  M = M_* + M_u + M_p
\]
according to whether
$c\neq0$ and $(a+d)^2\neq4\ell$ (generic),
or $c=0$ and $a\neq d$ (upper-triangular),
or $(a+d)^2=4\ell$ (parabolic).
Now we should recall the following results on counting lattice points.

\begin{lemma}\label{lemma:M_*}
  For any $z=x+iy\in \mathcal{F}(q)$, and any integer $L$ and $0<\delta<1$, we have
  \begin{equation*}
    \sum_{1\leq \ell\leq L}M_*(z,\ell,\delta,q) \ll_{q,\varepsilon} L^\varepsilon \left( L y^{-1} + L^{3/2}\delta^{1/2} + L^2\delta \right).
  \end{equation*}
\end{lemma}

\begin{proof}
  See \cite[Lemma 1.3]{templier2015hybrid}.
\end{proof}


\begin{lemma}\label{lemma:M_u}
  For any $z=x+iy\in \mathcal F(q)$, and any integer $L$ and $0<\delta<1$, the following estimate holds where $\ell_1,\ell_2$ run over primes:
  \begin{equation*}
    \sum_{1\leq \ell_1,\ell_2\leq L} M_u(z,\ell_1\ell_2,\delta,q)
    \ll_{q,\varepsilon} L^\varepsilon \left( L + L^3\delta^{1/2}y \right).
  \end{equation*}
\end{lemma}

\begin{proof}
  See \cite[Lemma 4.3]{templier2015hybrid} and \cite[Lemma 4.2]{HuangXu2017sup-norm}.
\end{proof}


\begin{lemma}\label{lemma:M_p}
  For any $z=x+iy\in \mathcal F(q)$, we have
  \begin{equation*}
    M_p(z,\ell,\delta,q)
    \ll_q \left(1+\ell^{1/2}\delta^{1/2}y+\ell^{3/4}\delta^{3/8}y^{-1/2}\right) \delta_\square(\ell),
  \end{equation*}
  where $\delta_\square(\ell) = 1,0$ depending on whether $\ell$ is a perfect square or not.
\end{lemma}

\begin{proof}
  See \cite[Lemma 13]{blomer2016sup-norm} and \cite[Lemma 4.3]{HuangXu2017sup-norm}.
\end{proof}

\section{Hecke L-functions $L(s,\phi_k)$} \label{sec:HeckeLfunction}

Let $\Xi_k$ be a Hecke Gr\"{o}ssencharacter with $k\geq1$.
Define its Hecke L-function
\begin{equation}\label{eqn:L-DS}
  L(s,\phi_k) := \sum_{0\neq\mathfrak{a}\subset \mathcal{O}} \frac{\Xi_k(\mathfrak{a})}{\N(\mathfrak{a})^s} = \prod_{\mathfrak{p}} \left(1-\frac{\Xi_k(\mathfrak{p})}{\N(\mathfrak{p})^s}\right)^{-1},
  \quad \textrm{$\Re(s)>1$}.
\end{equation}
We have
\[
  L(s,\phi_k) = \sum_{n=1}^{\infty} \frac{a_k(n)}{n^s},
  \quad \mathrm{with}\
  a_k(n)=\sum_{\N(\mathfrak{a})=n} \Xi_k(\mathfrak{a}),
  \quad \textrm{for $\Re(s)>1$}.
\]
It has the Euler product over rational primes, that is,
\begin{equation}\label{eqn:EP}
  L(s,\phi_k) = \prod_{p} \left(1-\frac{\alpha_k(p)}{p^s}\right)^{-1}
  \left(1-\frac{\beta_k(p)}{p^s}\right)^{-1},
  \quad \textrm{for $\Re(s)>1$},
\end{equation}
where the properties of $\alpha_k(p),\beta_k(p)$ depend on $p$ being split, inert, or ramified in $K=\mathbb{Q}(\sqrt{q})$.
In all cases, we have
\[
  |\alpha_k(p)|\leq 1 \quad \textrm{and}\quad
  |\beta_k(p)| \leq 1.
\]
The Rankin--Selberg L-function is (see e.g. \cite[\S5.11]{iwaniec2004analytic} after correcting the local factor at $q$)
\begin{equation}\label{eqn:RS}
  L(s,\phi_k\times \phi_k) = (1+q^{-s})\zeta(2s) \sum_{n=1}^{\infty} \frac{a_k(n)^2}{n^s}, \quad \textrm{$\Re(s)>1$},
\end{equation}
and the symmetric square L-function is (see e.g. \cite[\S5.12]{iwaniec2004analytic})
\begin{equation}\label{eqn:sym^2}
  L(s,\sym^2 \phi_k) = \frac{L(s,\phi_k\times \phi_k)}{L(s,\chi)} = \zeta(s) L(s,\phi_{2k}) .
\end{equation}


The logarithmic derivative of $L(s,\phi_k)$ is
\[
  \frac{L'(s,\phi_k)}{L(s,\phi_k)} = -\sum_{\mathfrak{p}} \log \N(\mathfrak{p}) \sum_{j=1}^{\infty} \frac{\Xi_k^j(\mathfrak{p})}{\N(\mathfrak{p})^{js}}
  , \quad \textrm{$\Re(s)>1$}.
\]
For $\Re(s)=\sigma>1$, by Landau's prime ideal theorem we have
\begin{equation}\label{eqn:L'/L-k}
  \Big|\frac{L'(s,\phi_k)}{L(s,\phi_k)}\Big|
  \ll \sum_{\mathfrak{p}} \log \N(\mathfrak{p}) \sum_{j=1}^{\infty} \frac{1}{\N(\mathfrak{p})^{j\sigma}}
  \ll \sum_{\mathfrak{p}}  \frac{\log \N(\mathfrak{p})}{\N(\mathfrak{p})^{\sigma}} +O(1) \ll \frac{\sigma}{\sigma-1}.
\end{equation}

The first result we will need is a nice zero-free region of $L(s,\phi_k)$.
\begin{lemma}\label{lemma:zero-free-region}
  Let $V=k+|t|$ with $k\geq1$. There exists a constant $c>0$ such that
  \[ L(\sigma+it,\phi_k) \neq 0 \quad \mathrm{for} \quad
  \sigma\geq 1-\frac{c}{(\log V)^{2/3}(\log\log V)^{1/3}}. \]
  Here the constant depends on the field $K$.
\end{lemma}

\begin{proof}
  See \cite[Theorem 2]{coleman1990zero}.
\end{proof}

%
%

Note that from \cite[Theorem 1 \& \S5]{coleman1990zero},we have
\begin{equation}\label{eqn:L<<}
  L(s,\phi_k) \ll (\log V)^{2/3},
\end{equation}
uniformly in the region
\[
  \sigma>1-\frac{c}{(\log V)^{2/3}}.
\]
We will also need the following upper bound of $L'(s,\phi_k)/L(s,\phi_k)$ in certain region.

\begin{lemma}\label{lemma:upper-bound-L'/L}
  Let $V=k+|t|$ with $k\geq1$. Then we have
  \[
    \frac{L'(s,\phi_k)}{L(s,\phi_k)} \ll (\log V)^{2/3}(\log\log V)^{1/3} \quad \textrm{and} \quad
    \frac{1}{L(s,\phi_k)} \ll (\log V)^{2/3}(\log\log V)^{1/3},
  \]
  uniformly for
  \[\sigma>1-\frac{c'}{(\log V)^{2/3}(\log\log V)^{1/3}}.\]
\end{lemma}

\begin{proof}
  Note that we have \eqref{eqn:L'/L-k} and \eqref{eqn:L<<}.
  The proof now follows essentially that of Theorem 3.11 in \cite{titchmarsh1986theory} with the notation as in \cite[Lemma 8]{coleman1990zero}, and taking
  \[
    \varphi(k,t)=c\log\log V, \quad
    \theta(k,t)=c' \left(\frac{\log\log V}{\log V}\right)^{2/3}. \qedhere
  \]
\end{proof}

\begin{remark}
  Similar results can be proved for general number fields by the same method together with the theorems in \cite{coleman1990zero}.
\end{remark}


\section{Amplifier}\label{sec:amplifier}

As in \cite[\S2.5]{HuangXu2017sup-norm}, we should construct an amplifier by just supporting on primes.
Let
\[
  \mathcal P = \{ p \ {\rm prime}: N\leq p\leq 2N\}.
\]
be a large set of primes, and define
\begin{equation}\label{eqn:x_n}
  x_n := \left\{\begin{array}{ll}
                  w(n/N)\log(n) a_{k}(n), & {\rm if}\ n\in \mathcal P, \\
                  0, &   {\rm otherwise},
                \end{array}\right.
\end{equation}
where $w$ is a fixed, compactly-support positive function on the positive reals, with
\begin{equation}\label{eqn:w}
  \supp(w)\subset[1,2],\quad 0\leq w(r)\leq 1, \quad
  {\rm and} \quad \int_{-\infty}^{\infty}w(r)\dd r \neq 0.
\end{equation}
Hence $y_\ell$ (defined in \eqref{eqn:y_ell})
satisfies:
\begin{equation}\label{eqn:y_ell<<}
  y_\ell \ll \left\{\begin{array}{ll}
                      N, & \textrm{if } \ell=1,\\
                      \log^2 N, & \textrm{if } \ell=\ell_1\ell_2 \ \textrm{with}\ \ell_1,\ell_2\in \mathcal P, \\
                      0, & \textrm{otherwise}.
                    \end{array}\right.
\end{equation}

Define
\begin{equation}\label{eqn:A_N}
  A_k(N) := A_{\phi_k} = \sum_{p=2}^{\infty} w(p/N) (\log p) a_k(p)^2,
\end{equation}
where $A_{\phi_k}$ is defined in \eqref{eqn:A}. Let
\begin{equation}\label{eqn:L(s)}
  L_k(s) := \frac{\zeta(s)L(s,\chi)L(s,\phi_{2k})}{ (1+q^{-s}) \zeta(2s) },
\end{equation}
for $\Re(s)>1$. 
By \eqref{eqn:RS} and \eqref{eqn:sym^2}, we derive
\[
  L_k(s) = \sum_{n=1}^\infty \frac{a_{k}(n)^2}{n^s}.
\]
Then, by the Euler product of $L_k(s)$ (see \eqref{eqn:EP}), we have
\begin{equation} \label{eqn:L'/L}
  \begin{split}
     - \frac{L_k'}{L_k}(s) & = \sum_{p}\log p \bigg(\frac{p^{-s}}{1-p^{-s}}+\frac{\chi(p) p^{-s}}{1-\chi(p)p^{-s}} \\
     & \hskip 60pt
     + \frac{\alpha_{2k}(p) p^{-s}}{1-\alpha_{2k}(p) p^{-s}}
     + \frac{\beta_{2k}(p) p^{-s}}{1-\beta_{2k}(p) p^{-s}}
     -\frac{2p^{-2s}}{1-p^{-2s}} \bigg) - \log q \frac{q^{-s}}{1+q^{-s}} \\
     & = \sum_{p}\log p \sum_{j=1}^{\infty} \left(\frac{1+\chi(p)^j+\alpha_{2k}(p)^j+\beta_{2k}(p)^j}{p^{js}}
     -\frac{2}{p^{2js}} \right) + \log q \sum_{j=1}^{\infty} \frac{(-1)^j}{q^{js}} \\
     & =: \sum_{n=1}^\infty \frac{b_k(n)}{n^s}, \qquad \Re(s)>1,
  \end{split}
\end{equation}
say.
By the Taylor expansion, we derive that
\begin{equation}\label{eqn:b(n)}
  b_k(n) = \left\{\begin{array}{ll}
    (\log p) (a_{k}(p))^2, & {\rm if}\ n=p,\\
    (\log p)b_{p,j}, & {\rm if}\ n=p^j,\ j\geq 2, \\
    0, & {\rm otherwise},
  \end{array}\right.
\end{equation}
with $|b_{p,j}|\leq 7$ for all $p$ prime and $j\geq 2$.
Define
\begin{equation*}
  B_k(N) := \sum_{n=1}^{\infty} w(n/N)b_k(n).
\end{equation*}
By \eqref{eqn:w}, \eqref{eqn:b(n)}, and the prime number theorem, we have
\begin{equation} \label{eqn:A_N2B_N}
  A_k(N) = B_k(N) + O(\sqrt{N}).
\end{equation}
So we can estimate $A_k(N)$ by the estimation of $B_k(N)$.

\begin{lemma}\label{lemma:amplifier}
  Let $k\geq1$ be large enough. Suppose that
  $$\log N \gg (\log k)^{2/3+\delta},$$
  for some fixed small $\delta>0$. Then
  \[
    A_k(N) = N \widetilde{w}(1)  + O\left(\frac{N}{(\log k)^{A}}\right),
  \]
  where
  $
  \widetilde{w}(s)=\int_0^\infty w(y)y^{s-1}\dd y
  $
  is the Mellin transform of $w$.
\end{lemma}

\begin{proof}
  By the Mellin transform and by \eqref{eqn:L(s)} and \eqref{eqn:L'/L}, we derive
  \[
    \begin{split}
      B_k(N) = & \frac{1}{2\pi i} \int_{(2)} N^s \widetilde{w}(s) \left( -\frac{L_k'}{L_k}(s) \right)\dd s \\
      = & \frac{1}{2\pi i} \int_{(2)} N^s \widetilde{w}(s)  \left( -\frac{\zeta'}{\zeta}(s)-\frac{L'(s,\chi)}{L(s,\chi)}
      -\frac{L'(s,\phi_{2k})}{L(s,\phi_{2k})}+\frac{\zeta'}{\zeta}(2s) + \frac{\log q}{1+q^s} \right)\dd s.
    \end{split}
  \]
  Next we move the contour to the left, to one along the straight line segments $L_1,L_2,L_3$ defined by
  $$
    L_1=\left\{ \sigma_0+it: |t|\leq k \right\}, \quad
    L_2=\left\{ 1+it: |t|\geq k \right\},
  $$
  and the short horizontal segments
  $$
    L_3=\left\{ \sigma\pm ik: \sigma_0 \leq \sigma \leq 1 \right\},
  $$
  where $\sigma_0 = 1-\frac{c}{(\log k)^{2/3+\delta/2}}$ with $c$ being a small positive number such that $L_k(s)$ is zero-free
  on the boundary and right side of $L_1\cup L_2\cup L_3$.
  \begin{center}
    \begin{tikzpicture}
      \draw[black,->] (-1,0) -- (2.5,0)
            node[xshift=0.2cm] {$\sigma$};
      \draw[black,->] (0,-2) -- (0,2)
            node[yshift=0.2cm] {$t$};
      \filldraw[black] (1.5,0) circle (0.04) node at (1.7,-0.25) {\small $1$};
      \filldraw[black] (0,1.5) circle (0.04) node at (-0.25,1.5) {\small $k$};
      \filldraw[black] (0,-1.5) circle (0.04) node at (-0.35,-1.5) {\small $-k$};
      \filldraw[black] (1.2,0) circle (0.04) node at (0.95,-0.25) {\small $\sigma_0$};
      \draw [dashed] (1.5,-1.5) -- (1.5,1.5);
      \draw [dashed] (0,1.5) -- (1.2,1.5);
      \draw [dashed] (0,-1.5) -- (1.2,-1.5);
      \draw [thick] (1.5,-2.2) -- (1.5,-1.5) -- (1.2,-1.5) -- (1.2,1.5) -- (1.5,1.5) -- (1.5,2.2) ;
    \end{tikzpicture}
  \end{center}
  By \cite[Theorem 8.29]{iwaniec2004analytic}, we may also use
  the Vinogradov--Korobov bound $\zeta'(s)/\zeta(s)\ll (\log |t|)^{2/3}(\log\log |t|)^{1/3}$
  in this region.
  The integrals along the line segments $L_2$ and $L_3$ are trivially bounded by $O(k^{-100})$ by the rapid decay of $\widetilde{w}$.
  Together with Lemmas \ref{lemma:zero-free-region} and \ref{lemma:upper-bound-L'/L}, the new line $L_1$ gives an amount that is certainly
  \begin{equation*}
    \ll N (\log k) \exp\left( -\frac{\log N}{(\log k)^{2/3+2\delta/3}} \right) \ll \frac{N}{(\log k)^{100}}.
  \end{equation*}
  Now we need to analyze the residue of the pole from $-\zeta'(s)/\zeta(s)$.
  The residue at $s=1$ contributes $N \widetilde{w}(1)$.
  Hence, by \eqref{eqn:A_N2B_N}, we prove the lemma.
\end{proof}


\section{Sup-norm}\label{sec:supnorm}

We will use the amplified pre-trace formula \eqref{eqn:APTF}.
To obtain upper bounds we use a test function $h(r)$
which is localized for $r$ near $T$, with $T\geq 2$ being a parameter.
We need a suitable point-pair kernel and the coming estimate.
\begin{lemma}\label{lemma:k_T}
  For all $T\geq1$, there is a point-pair kernel $k_T\in C_c^\infty([0,\infty))$, supported on
  $[0,1]$, which satisfies the following properties:
  \begin{itemize}
    \item [(i)]   The spherical transform $h_T(r)$ is positive for all $r\in \mathbb R\cup i\mathbb R$,
    \item [(ii)]  For all $T\leq r\leq T+1$, $h_T(r)\gg 1$,
    \item [(iii)] For all $u\geq 0$, $|k_T(u)|\leq T$,
    \item [(iv)]  For all $T^{-2}\leq u\leq 1$, $|k_T(u)|\leq \frac{T^{1/2}}{u^{1/4}}$.
  \end{itemize}
\end{lemma}

\begin{proof}
  See Templier \cite[Lemma 2.1]{templier2015hybrid}.
\end{proof}

%

Hence, by \eqref{eqn:APTF}, we have
\begin{equation}\label{eqn:Iwaniec-Sarnak}
  |A_{k}(N)|^2 \frac{|\phi_k(z)|^2}{\|\phi_k\|_2^2}
  \ll \sum_{\ell} \frac{|y_\ell|}{\sqrt{\ell}} \sum_{\gamma\in \mathcal M(\ell,q)} |k_{T_k}(\gamma z,z)|,
\end{equation}
with $x_n$ being defined as in \eqref{eqn:x_n}.

Now using Lemmas \ref{lemma:M_*}, \ref{lemma:M_u}, and \ref{lemma:M_p}, the same argument as in \cite[\S5]{HuangXu2017sup-norm} gives
\begin{equation}\label{eqn:sum_l<<}
  \sum_{\ell} \frac{|y_\ell|}{\sqrt{\ell}} \sum_{\gamma\in \mathcal M(\ell,q)} |k_{T_k}(\gamma z,z)|
  \ll_q (NT_k)^\varepsilon (NT_k+N^3T_k^{1/2}+N^2T_k^{1/2}y).
\end{equation}

From the bound via Fourier expansion in Lemma \ref{lemma:Bound_FE}, we can assume without loss of
generality when establishing Theorem \ref{thm:sup-norm} that
\begin{equation*}
  y\ll T_k^{1/4}.
\end{equation*}
Combining Lemma \ref{lemma:amplifier} with
\eqref{eqn:Iwaniec-Sarnak} and \eqref{eqn:sum_l<<}, we obtain that
\begin{equation*}
  \frac{|\phi_k(z)|^2}{\|\phi_k\|_2^2} \ll (NT_k)^\varepsilon \left( T_k N^{-1} + N T_k^{1/2} + T_k^{1/2}y \right).
\end{equation*}
By choosing $N := T_k^{1/4}$, we obtain $\phi_k(z) \ll T_k^{3/8+\varepsilon}\|\phi_k\|_2$,
as claimed in Theorem \ref{thm:sup-norm}.

\begin{remark}
  In \cite{jung2016quantitative}, Jung proved the same sup-norm bounds for almost all Hecke--Maass forms in short intervals by using a sharp estimate on the variance of the shifted convolution sums.
  It's possible to prove that for a fixed compact set $\Omega\in\mathbb{X}$,
  \[
    \sup_{z\in\Omega}|\phi_k(z)| \ll_{\Omega,\varepsilon} T_k^{3/8+\varepsilon} \|\phi_k\|_2,
  \]
  by using quantitative QUE and shifted convolution sums (see \cite{jung2016quantitative}), which may be proved by subconvexity bounds for Rankin--Selberg L-functions (see e.g. 
  \cite{LiuYe2002subconvexity} and 
  \cite{LauLiuYe2006new}).
\end{remark}

%

%


\section{Second moment of Hecke L-functions}\label{sec:secondmoment}

In this section, our main goal is to prove
the second moment of Hecke L-functions (Theorem \ref{thm:second_moment}).
To do this, we will first give an approximate functional equation of $L(1/2+it,\phi_k)$, discuss gaps between angles of integral ideals, and then prove a large sieve inequality for our family of Hecke Gr\"{o}ssencharacters.

\subsection{The approximate functional equation}

Define the completed L-function by
\[
  \Lambda(s,\phi_k) := L_\infty(s,\phi_k)L(s,\phi_k),
\]
with
$$
  L_\infty(s,\phi_k) := \Big(\frac{\sqrt{q}}{\pi}\Big)^s \Gamma\Big(\frac{s+iT_k}{2}\Big)\Gamma\Big(\frac{s-iT_k}{2}\Big).
$$
By Hecke \cite{hecke1920eine}, we have the functional equation
\[
  \Lambda(s,\phi_k) = \omega_k \Lambda(1-s,\phi_k),
\]
where $|\omega_k|=1$.
We will need the following approximate functional equation.
Note that we remove the dependence of $k$ for the coefficients $c(\mathfrak{a})$ and the length of the sum.

\begin{lemma}\label{lemma:AFE}
  Let $T\geq2$ be sufficiently large, $T\leq T_k\leq 2T$, and $|t|\leq T^{1-\varepsilon}$. For any $\varepsilon>0$, we have
  \begin{equation*}
    L(1/2+it,\phi_k) \ll_{q,\varepsilon} T^\varepsilon  \int_{\varepsilon-iT^{\varepsilon}}^{\varepsilon+iT^{\varepsilon}} \Big|\sum_{\substack{\mathfrak{a}\subset \mathcal{O}\\\N(\mathfrak{a})\leq T^{1+\varepsilon}}} \frac{\Xi_k(\mathfrak{a})}{\N(\mathfrak{a})^{1/2+it+s}}\Big| |\dd s| + 1.
  \end{equation*}
\end{lemma}

\begin{proof}
  By Iwaniec--Kowalski \cite[\S5.2]{iwaniec2004analytic}, we have 
  \[
    L(1/2+it,\phi_k)
    =  \sum_{\mathfrak{a}\subset \mathcal{O}} \frac{\Xi_k(\mathfrak{a})}{\N(\mathfrak{a})^{1/2+it}}
    V_{k,t}(\N(\mathfrak{a}))
    + \eta_k \sum_{\mathfrak{a}\subset \mathcal{O}} \frac{\Xi_k(\mathfrak{a})}{\N(\mathfrak{a})^{1/2-it}}
    V_{k,-t}(\N(\mathfrak{a})),
  \]
  where $|\eta_k|=1$ and 
  \begin{equation}\label{eqn:V(y)}
    V_{k,t}(y) := \frac{1}{2\pi i} \int_{(2)} \frac{L_\infty(1/2+it+s,\phi_k)}{L_\infty(1/2+it,\phi_k)} y^{-s} G(s) \frac{\dd s}{s}, 
  \end{equation}
  with $G(s)=e^{s^2}.$

  Now we want to truncate the sums and remove the dependence of $k$ for weight functions.
  Note that $T_k\asymp T$ and $|t|\ll T^{1-\varepsilon}$.
  By Stirling's formula, we can truncate the sums at $\N(\mathfrak{a})\leq T^{1+\varepsilon}$ with negligible errors by shifting the contour of integration to the right.
  By the triangle inequality, we only need to consider the first sum.

  From now on, we assume that $\N(\mathfrak{a})\leq T^{1+\varepsilon}$. In \eqref{eqn:V(y)}, we can move the line of integration to $\Re(s)=\varepsilon$, and then truncate at $\Im(s)\ll T^{\varepsilon}$ with a negligible error again, getting
  \[
    V_{k,t}(y) = \frac{1}{2\pi i} \int_{\varepsilon-iT^{\varepsilon}}^{\varepsilon+iT^{\varepsilon}} \frac{L_\infty(1/2+it+s,\phi_k)}{L_\infty(1/2+it,\phi_k)} y^{-s} G(s) \frac{\dd s}{s} + O(T^{-A}).
  \]
  Since by Stirling's formula we have
  \[
    \frac{L_\infty(1/2+it+s,\phi_k)}{L_\infty(1/2+it,\phi_k)}
    \ll_\varepsilon (qT)^\varepsilon,
  \]
  we obtain
  \[
    \begin{split}
    L(1/2+it,\phi_k) &
    \ll \Big| \sum_{\mathfrak{a}\subset \mathcal{O}} \frac{\Xi_k(\mathfrak{a})}{\N(\mathfrak{a})^{1/2+it}}
    V_{k,t}(\N(\mathfrak{a})) \Big|
    \\
    & \ll_{q,\varepsilon} T^\varepsilon  \int_{\varepsilon-iT^{\varepsilon}}^{\varepsilon+iT^{\varepsilon}} \Big|\sum_{\substack{\mathfrak{a}\subset \mathcal{O}\\\N(\mathfrak{a})\leq T^{1+\varepsilon}}} \frac{\Xi_k(\mathfrak{a})}{\N(\mathfrak{a})^{1/2+it+s}}\Big| |\dd s| + 1,
    \end{split}
  \]
  as claimed.
\end{proof}

\subsection{Gaps between angles} \label{subsec:angles}
Note that if $\alpha$ is a generator of an integral ideal $\mathfrak{a}$, then so is $\pm \epsilon_q^n \alpha$, for any $n\in\mathbb{Z}$.
So for $\mathfrak{a}=(\alpha)$ with $\alpha=a+b\omega_q$, we can define an angle $t_\mathfrak{a}\ (\mod 2\log \epsilon_q)$ of $\mathfrak{a}$ by
$e^{t_\mathfrak{a}} = |\frac{\alpha}{\tilde{\alpha}}|=|\frac{a+b\omega_q}{a+b\tilde\omega_q}|$.

\begin{lemma}\label{lemma:angle-gap}
  Let $\|x\|_{\mathbb Z}:=\min_{n\in\mathbb Z}|x-n|$. Then we have
  \begin{itemize}
    \item [i)]  If $t_\mathfrak a \neq 0 \ \mod 2\log\epsilon_q$, then 
    	\begin{equation*}
		  \big\| \frac{t_\mathfrak a}{2\log\epsilon_q} \big\|_\mathbb Z
		  \gg
		  \frac{1}{\sqrt{\N(\mathfrak a)}};
		\end{equation*}
    \item [ii)] If $t_\mathfrak a \neq t_\mathfrak b \ \mod 2\log\epsilon_q$, then 
    	\[
		  \big\| \frac{1}{2\log\epsilon_q}(t_\mathfrak a - t_\mathfrak b) \big\|_\mathbb Z
		  \gg
		   \frac{1}{\sqrt{\N(\mathfrak a)\N(\mathfrak b)}}.
		\]
  \end{itemize}
\end{lemma}

\begin{proof}
  (i) We can assume
  $\|\frac{t_\mathfrak a}{2\log\epsilon_q}\|_{\mathbb Z} \leq 1/10$,
  otherwise we already prove the lemma.
  Write $\epsilon=\epsilon_q$. Note that we can choose $t_\mathfrak{a}\in(-\log \epsilon,\log \epsilon]$, so
  in order to give a lower bound for $\| \frac{t_\mathfrak a}{2\log\epsilon} \|_\mathbb Z$,  we only need to consider
  $| \frac{t_\mathfrak a}{2\log\epsilon}|$.
  Since we have the action of $U=\pm\epsilon^{\mathbb{Z}}$, we can choose a generator $\alpha=a+b\omega_q$ of $\mathfrak{a}$ such that
  \[
    e^{t_\mathfrak{a}} = e^{t_\alpha} := |\frac{a+b\omega_q}{a+b\tilde\omega_q}|
    =\frac{(a+b\omega_q)^2}{\N(\mathfrak{a})}
    =\frac{\N(\mathfrak{a})}{(a+b\tilde\omega_q)^2} \in (\epsilon^{-1},\epsilon),
  \]
  and at least one of $a+\frac{b}{2}$ and $b$ is positive.
  Hence
  \begin{equation}\label{eqn:ratio1}
    (a+\frac{b}{2})^2 + (\frac{b\sqrt{q}}{2})^2 = \frac{1}{2} \left((a+b\omega_q)^2 + (a+b\tilde\omega_q)^2 \right) \asymp_q \N(\mathfrak{a}).
  \end{equation}

  If $\N(\alpha)>0$, we can assume $a+\frac{b}{2}>0$ (replacing $\alpha$ by $-\alpha$ if necessary). Note that $\N(\alpha) = (a+\frac{b}{2})^2-\frac{b^2 q}{4}>0$. By $t_\mathfrak a \neq 0 \ \mod 2\log\epsilon$, we have
  \begin{equation*}
    \frac{\sqrt{q}}{2} \leq \frac{|b|\sqrt{q}}{2} < a+\frac{b}{2}.
  \end{equation*}
  Thus by \eqref{eqn:ratio1}, we have
  \[
    a+\frac{b}{2} \asymp \sqrt{\N(\mathfrak{a})}.
  \]
  Note that
  $e^{t_\mathfrak{a}/2} = \frac{a+b\omega_q}{\sqrt{\N(\mathfrak{a})}}$ and
  $e^{-t_\mathfrak{a}/2} = \frac{a+b\tilde\omega_q}{\sqrt{\N(\mathfrak{a})}}$,
  which give us
  \[
    |t_\mathfrak a /2| \geq |\tanh t_\mathfrak a/2|
    = \frac{|b|\sqrt{q}}{2(a+\frac{b}{2})} \gg \frac{1}{\sqrt{\N(\mathfrak{a})}}.
  \]
  Thus
  \begin{equation*}
    \min_{t_\mathfrak a\neq 0 \mod 2\log\epsilon}
    \| \frac{t_\mathfrak a}{2\log\epsilon} \|_\mathbb Z
    \gg
    \frac{1}{\sqrt{\N(\mathfrak a)}}.
  \end{equation*}

  Similarly, if $\N(\alpha)<0$, we can assume $b>0$. Note that $-\N(\mathfrak{a})=\N(\alpha) = (a+\frac{b}{2})^2-\frac{b^2 q}{4}<0$. By $t_\mathfrak a \neq 0 \ \mod 2\log\epsilon$, we have
  \begin{equation*}
    \frac{1}{2}\leq |a+\frac{b}{2}| < \frac{b\sqrt{q}}{2}.
  \end{equation*}
  Thus by \eqref{eqn:ratio1} again, we have
  \[
    \frac{b\sqrt{q}}{2} \asymp \sqrt{\N(\mathfrak{a})}.
  \]
  Note that
  $e^{t_\mathfrak{a}/2} = \frac{a+b\omega_q}{\sqrt{\N(\mathfrak{a})}}$ and
  $e^{-t_\mathfrak{a}/2} = - \frac{a+b\tilde\omega_q}{\sqrt{\N(\mathfrak{a})}}$,
  which give us
  \[
    |t_\mathfrak a /2| \geq |\tanh t_\mathfrak a/2|
    = \frac{2(a+\frac{b}{2})}{|b|\sqrt{q}} \gg \frac{1}{\sqrt{\N(\mathfrak{a})}}.
  \]
  Thus
  \begin{equation*}
    \min_{t_\mathfrak a\neq 0 \mod 2\log\epsilon}
    \| \frac{t_\mathfrak a}{2\log\epsilon} \|_\mathbb Z
    \gg
    \frac{1}{\sqrt{\N(\mathfrak a)}}.
  \end{equation*}

  (ii) Consider the ideal
  $\mathfrak{c} = \mathfrak{a} \tilde{\mathfrak b}$ in (i),
  where $\tilde{\mathfrak{b}}$ is the conjugate ideal of $\mathfrak{b}$.
  Since $t_\mathfrak a \neq t_\mathfrak b \ \mod 2\log\epsilon_q$, we have $t_\mathfrak c = t_\mathfrak{a} - t_\mathfrak b  \neq 0 \ \mod 2\log\epsilon_q$. By (i), we have
  \[
    \big\| \frac{1}{2\log\epsilon_q} (t_\mathfrak a - t_\mathfrak b)\big\|_\mathbb Z
    = \big\| \frac{t_\mathfrak c}{2\log\epsilon_q} \big\|_\mathbb Z \gg
    \frac{1}{\sqrt{\N(\mathfrak c)}} = \frac{1}{\sqrt{\N(\mathfrak a)\N(\mathfrak b)}}.
  \]
  This completes the proof of our lemma.
\end{proof}

According to the above proof, we can divide the integral ideals into two classes. We say $\mathfrak{a}\in\mathfrak{A}_1$ if we can choose a generator $\alpha$ such that 
$t_\alpha\in (-\log \epsilon_q,\log \epsilon_q]$ and $\N(\alpha)>0$.
And we say $\mathfrak{a}\in\mathfrak{A}_2$ if we can choose a generator $\alpha$ such that 
$t_\alpha\in (-\log \epsilon_q,\log \epsilon_q]$ and $\N(\alpha)<0$.
Note that $\mathfrak{A}_1\cap\mathfrak{A}_2=\varnothing$, and $\mathfrak{A}_1\cup\mathfrak{A}_2$ is equal to the set of all integral ideals. We say $\mathfrak{a}$ is primitive if for our choice of $\alpha=a+b\omega_q$ we have $\gcd(a,b)=1$.

\subsection{A large sieve inequality}

In order to bound the second moment of Hecke L-functions, we will use the following large sieve inequality.

\begin{lemma}\label{lemma:large_sieve}
  Let $N\geq2$ and $K\geq2$ be given.
  Then for any sequence $\{c(\mathfrak a)\}$, we have
  \begin{equation*}
    \begin{split}
      & \sum_{1\leq k \leq K}  \Big|\sum_{\N(\mathfrak a)\leq N} \Xi_k(\mathfrak a) c(\mathfrak a) \N(\mathfrak a)^{-1/2}\Big|^{2}
      \ll
      (K+N)  \sum_{j=1,2} \sum_{\substack{\mathfrak a \ \mathrm{primitive} \\ \mathfrak{a}\in\mathfrak{A}_j \\ \N(\mathfrak a)\leq N}}
    	\Big(\sum_{\substack{t_\mathfrak b = t_\mathfrak a \\ \mathfrak{b}\in\mathfrak{A}_j \\ \N(\mathfrak b)\leq N}}
			|c(\mathfrak b)|\N(\mathfrak b)^{-1/2}\Big)^2.
    \end{split}
  \end{equation*}
\end{lemma}

\begin{proof}
  Let
  \[
    S_j(u) := \sum_{\substack{1\leq \N(\mathfrak a)\leq u \\ \mathfrak{a}\in\mathfrak{A}_j}} \Xi_k(\mathfrak a)
    c(\mathfrak a) \N(\mathfrak a)^{-1/2}, \quad j=1,2.
  \]
  By the triangle inequality, we have
  \[
    \sum_{1\leq k \leq K}  \Big|\sum_{\N(\mathfrak a)\leq N} \Xi_k(\mathfrak a) c(\mathfrak a) \N(\mathfrak a)^{-1/2}\Big|^{2}
    \ll \sum_{j=1,2}\sum_{1\leq k \leq K} \Big| S_j(N) \Big|^2.
  \]
  So we only need to consider the following summation
  \begin{equation*}
    \sum_{1\leq k \leq K} \Big| S_j(N) \Big|^2
    = \sum_{1\leq k\leq K} \Big| \sum_{\substack{\N(\mathfrak a)\leq N \\ \mathfrak{a}\in\mathfrak{A}_j}}
    c(\mathfrak a)\N(\mathfrak a)^{-1/2} e^{\frac{\pi i k t_\mathfrak a}{\log \epsilon_q}} \Big|^2.
  \end{equation*}
  We rearrange the innermost sum, getting
  \[
    \sum_{1\leq k \leq K} \Big| S_j(N) \Big|^2
    = \sum_{1\leq k\leq K} \Big| \sum_{\substack{\mathfrak a \ \textrm{primitive} \\ \mathfrak{a}\in\mathfrak{A}_j \\ \N(\mathfrak a)\leq N}}
    \Big(\sum_{\substack{t_\mathfrak b = t_\mathfrak a \\ \mathfrak{b}\in\mathfrak{A}_j  \\ \N(\mathfrak b) \leq N}}
    c(\mathfrak b)\N(\mathfrak b)^{-1/2}\Big)
    e^{\frac{\pi i k t_\mathfrak b}{\log \epsilon_q}} \Big|^2.
  \]

  Note that for two different primitive integral ideals $\mathfrak a_1$
  and $\mathfrak a_2$ in the same set $\mathfrak{A}_j$ with $\N(\mathfrak a_l)\leq N$, $l=1,2$,
  we have $t_{\mathfrak a_1} \neq t_{\mathfrak a_2} \ \mod 2\log \epsilon_q$.
  Indeed, we can choose a generator $\alpha_l=a_l+b_l\omega_q$ of $\mathfrak{a}_l$ such that
  $t_{\alpha_l}\in(-\log \epsilon_q,\log \epsilon_q]$ and $\N(\alpha_1)\N(\alpha_2)>0$. If we assume $t_{\mathfrak a_1} \equiv t_{\mathfrak a_2} \ \mod 2\log \epsilon_q$, then $t_{\alpha_1}=t_{\alpha_2}$. Hence $\frac{a_1+b_1\omega_q}{a_1+b_1\tilde\omega_q}=\frac{a_2+b_2\omega_q}{a_2+b_2\tilde\omega_q}$.
  By the coefficient of $\sqrt{q}$ we get $\frac{a_1}{b_1}=\frac{a_2}{b_2}$. This contradicts to the condition that $\mathfrak a_1$ and $\mathfrak a_2$ are two different primitive integral ideals.

  Now by Lemma \ref{lemma:angle-gap}, we have
  $\big\| \frac{1}{2\log\epsilon_q}(t_\mathfrak a - t_\mathfrak b) \big\|_\mathbb Z
		  \gg
		   \frac{1}{N}$.
  By Montgomery--Vaughan's inequality \cite[eq. (2.3)]{montgomery1973large}, we have
  \[
    \sum_{1\leq k \leq K} \Big| S_j(N) \Big|^2
    \ll  (K+N) \sum_{\substack{\mathfrak a \ \mathrm{primitive} \\ \mathfrak{a}\in\mathfrak{A}_j \\ \N(\mathfrak a)\leq N}}
    	\Big(\sum_{\substack{t_\mathfrak b = t_\mathfrak a \\ \mathfrak{b}\in\mathfrak{A}_j \\ \N(\mathfrak b)\leq N}}
			|c(\mathfrak b)|\N(\mathfrak b)^{-1/2}\Big)^2.
  \]
  This proves our lemma.
\end{proof}

\subsection{Completion of the proof}

As in the proof of Lemma \ref{lemma:large_sieve}, for two integral ideals $\mathfrak{a}$ and $\mathfrak{b}$ satisfy that
\[
  t_\mathfrak b = t_\mathfrak a , \quad
  \mathfrak{a},\mathfrak{b}\in\mathfrak{A}_j, \quad
  \N(\mathfrak a) \leq N, \quad
  \N(\mathfrak b)\leq N, \quad
  \mathfrak{a} \ \textrm{primitive},
\]
we obtain $\mathfrak{b}=(m)\mathfrak{a}$ for some rational integer $m\geq1$.
Hence for a primitive ideal $\mathfrak{a}$ such that $\mathfrak{a}\in\mathfrak{A}_j$ and $\N(\mathfrak a)\leq N$, if $c(\mathfrak{b})\ll N^{\varepsilon}$ for all $\mathfrak{b}$, we have
\[
  \sum_{\substack{t_\mathfrak b = t_\mathfrak a \\ \mathfrak{b}\in\mathfrak{A}_j \\ \N(\mathfrak b)\leq N}}
			|c(\mathfrak b)|\N(\mathfrak b)^{-1/2}
  \ll N^\varepsilon \N(\mathfrak{a})^{-1/2}.
\]
Together with Lemmas \ref{lemma:AFE} and \ref{lemma:large_sieve}, we get
\[
  \begin{split}
    \frac{1}{K} \sum_{K<k\leq 2K} |L(1/2+it,\phi_k)|^2
    & \ll_{q,\varepsilon} K^{\varepsilon} +
    K^{\varepsilon-1}
    \int_{\varepsilon-iT^{\varepsilon}}^{\varepsilon+iT^{\varepsilon}} \sum_{K<k\leq 2K} \Big|\sum_{\substack{\mathfrak{a}\subset \mathcal{O}\\\N(\mathfrak{a})\leq K^{1+\varepsilon}}} \frac{\Xi_k(\mathfrak{a})}{\N(\mathfrak{a})^{1/2+it+s}}\Big|^2 |\dd s|
     \\
    & \ll_{q,\varepsilon} K^{\varepsilon} +
     K^{\varepsilon}\sum_{j=1,2} \sum_{\substack{\mathfrak a \ \mathrm{primitive} \\ \mathfrak{a}\in\mathfrak{A}_j \\ \N(\mathfrak a)\leq K^{1+\varepsilon}}}
    	\Big(\sum_{\substack{t_\mathfrak b = t_\mathfrak a \\ \mathfrak{b}\in\mathfrak{A}_j \\ \N(\mathfrak b)\leq K^{1+\varepsilon}}}
			\N(\mathfrak b)^{-1/2-\varepsilon}\Big)^2 \\
    & \ll_{q,\varepsilon} K^{\varepsilon} +
     K^{\varepsilon} \sum_{\N(\mathfrak a) \leq K^{1+\varepsilon}} \N(\mathfrak a)^{-1}
     \ll_{q,\varepsilon}  K^{\varepsilon}.
  \end{split}
\]
This completes the proof of Theorem \ref{thm:second_moment}.


\section{Nodal domains}\label{sec:nodaldomains}

We now prove Theorem \ref{thm:nodal_domain}. We first recall the following sharp lower bound for the $L^2$-norm of restriction to any fixed geodesic segment.

\begin{lemma}\label{lemma:rL2>>}
  Let $\beta\subset \delta=\left\{iy:y>0\right\}$ be any fixed compact geodesic segment. Then
  \[
    \|\phi_k\|^2_{L^2\left(\beta\right)}
    := \int_{\beta} |\phi_k(z)|^2 \dd s
    \gg_{\beta} \|\phi_k\|_2^2.
  \]
\end{lemma}

\begin{proof}
  By \cite[Theorem 1.3]{LiuYe2002subconvexity}, we know that QUE holds for the sequence of eigenfunctions $\{\phi_k\}$. Then by \cite[Corollary 3.2]{JangJung2018quantum}, the lemma is proved.
\end{proof}

Fix a geodesic segment $\beta \subset \left\{iy:y>0\right\}$,
and assume that it is given by $\{iy:a<y<b\}$.
Let 
\[
  M\left(\phi_k\right) := \frac{1}{\|\phi_k\|_2} \sup_{a < \alpha_1 <\alpha_2 <b} \left|\int_{\alpha_1}^{\alpha_2} \phi_k\left(iy\right) \frac{\dd y}{y}\right|.
\]
We will prove the following lemma.
\begin{lemma} \label{lemma:M<<}
  We have
  \[
    \sum_{K<k\leq 2K} M(\phi_k)^2
    \ll_{q,\varepsilon} K^\varepsilon .
  \]
  Therefore, we have
  \[ M(\phi_k) \ll T_k^{-1/2+\varepsilon} , \]
  for all but $O_{q,\varepsilon}(K^{1-\varepsilon})$ forms in $\{\phi_k:K<k\leq 2K\}$.
\end{lemma}

\begin{proof}
  Let $T=\pi K/\log \epsilon_q$. By Parseval's theorem, Stirling's formula, and \eqref{eqn:L^2}, we have (see e.g. \cite[\S6.2.3]{GhoshReznikovSarnak2013nodal})
  \begin{equation}\label{eqn:M=M1+M2}
    \begin{split}
       M\left(\phi_k\right) & \ll_{q,\varepsilon} T_k^{-1/4+\varepsilon} \int_{0}^{2T_k} \left|L\left(\frac{1}{2}+it,\phi_k\right)\right| \left(1+|t-T_k|\right)^{-1/4} \min\left(1,\frac{1}{t}\right) \dd t + e^{-cT_k}  \\
       & \ll_{q,\varepsilon} M_1(\phi_k) + M_2(\phi_k) + e^{-cT_k},
    \end{split}
  \end{equation}
  where
  \[
    \begin{split}
       M_1(\phi_k) & := T^{-1/2+\varepsilon} \int_0^{T^{1-\varepsilon}}\left|L\left(\frac{1}{2}+it,\phi_k\right)\right| \frac{ \dd t}{1+t}  , \\
       M_2(\phi_k) & := T^{-5/4+\varepsilon} \int_{T^{1-\varepsilon}}^{2T_k}\left|L\left(\frac{1}{2}+it,\phi_k\right)\right| \left(1+|t-T_k|\right)^{-1/4} \dd t.
    \end{split}
  \]

  We first deal with $M_1(\phi_k)$.
  By Theorem \ref{thm:second_moment}, we have
  \begin{equation}\label{eqn:M_1<<}
    \begin{split}
       \sum_{K<k\leq 2K} M_1(\phi_k)^2 & =
       T^{-1+2\varepsilon} \sum_{K<k\leq 2K} \bigg( \int_0^{T^{1-\varepsilon}} \left|L\left(\frac{1}{2}+it,\phi_k\right)\right| \frac{ \dd t}{1+t} \bigg)^2 \\
         & \ll T^{-1+2\varepsilon}  \int_0^{T^{1-\varepsilon}} \bigg( \sum_{K<k\leq 2K} \left|L\left(\frac{1}{2}+it,\phi_k\right)\right|^2 \bigg) \frac{ \dd t}{1+t}  \\
         & \ll_{q,\varepsilon} T^{\varepsilon}.
    \end{split}
  \end{equation}

  Now we consider $M_2(\phi_k)$.
  By Cauchy--Schwarz inequality, we have
  \[
    M_2(\phi_k) \ll T_k^{-3/4+\varepsilon} \bigg( \int_{T^{1-\varepsilon}}^{2T_k}\left|L\left(\frac{1}{2}+it, \phi_k\right)\right|^2 \left(1+|t-T_k|\right)^{-1/2} \dd t \bigg)^{1/2}.
  \]
  Now by using the approximate functional equation (see \cite[p. 1547]{GhoshReznikovSarnak2013nodal}), we have
  \[
    T_k^{-1/2+\varepsilon} \int_{0}^{2T_k} \left|L\left(\frac{1}{2}+it, \phi_k\right)\right|^2 \left(1+|t-T_k|\right)^{-1/2} \dd t \ll_{q,\varepsilon} T_k^\varepsilon.
  \]
  Hence, we obtain
  \begin{equation*}
    M_2(\phi_k) \ll_{q,\varepsilon} T_k^{-1/2+\varepsilon}.
  \end{equation*}
  Thus by $T\asymp T_k \asymp K$, we have
  \begin{equation}\label{eqn:M_2<<}
    \sum_{K<k\leq 2K} M_2(\phi_k)^2 \ll_{q,\varepsilon} K^{\varepsilon}.
  \end{equation}
  Combining \eqref{eqn:M=M1+M2}, \eqref{eqn:M_1<<}, and \eqref{eqn:M_2<<}, we prove the first claim in the lemma.

  Note that
  \[
    \sum_{\substack{K<k \leq 2K \\ M(\phi_k)\gg T_k^{-1/2+\varepsilon}}} 1
    \ll T^{1-2\varepsilon} \sum_{\substack{K<k \leq 2K \\ M(\phi_k)\gg T_k^{-1/2+\varepsilon}}} M(\phi_k)^2
    \leq T^{1-2\varepsilon} \sum_{K<k \leq 2K } M(\phi_k)^2
    \ll_{q,\varepsilon} K^{1-\varepsilon}.
  \]
  This proves our lemma.
\end{proof}

Let $S_\beta(\phi_k)$ be the number of sign changes of $\phi_k$ along $\beta$. Denote by $a<\xi_1(\phi_k) < \xi_2(\phi_k) < \ldots < \xi_{S_\beta(\phi_k)}(\phi_k)<b$ the zeros of $\phi_k(iy)$ on the interval $(a,b)$ where $\phi_k(iy)$ changes sign. Put $\xi_0(\phi_k) = a$ and $\xi_{S_\beta(\phi_k)+1}=b$. Then we have
\begin{align*}
  \int_a^b |\phi_k(iy)| \frac{\dd y}{y} &
    = \sum_{j=1}^{S_\beta(\phi_k)+1} \left|\int_{\xi_{j-1}(\phi_k)}^{\xi_{j}(\phi_k)}\phi_k(iy) \frac{\dd y}{y} \right|
  \leq \sum_{j=1}^{S_\beta(\phi_k)+1} M(\phi_k) \|\phi_k\|_2,
\end{align*}
hence
\[
  \|\phi_k\|_{L^1\left(\beta\right)} \leq M\left(\phi_k\right)  \|\phi_k\|_2 \left(S_\beta\left(\phi_k\right)+1\right),
\]
By \cite[Eq. (6) and Remark 2.2]{GhoshReznikovSarnak2013nodal} and \cite[\S 2.5]{GhoshReznikovSarnak2017nodal}, we have
\[
  N^\beta \left(\phi_k\right) \geq \frac{1}{2} S_\beta\left(\phi_k\right) - g +1,
\]
where $g$ is the genus of the surface $\Gamma_0(q)\backslash\mathbb{H}$.
(Ghosh--Reznikov--Sarnak \cite{GhoshReznikovSarnak2013nodal} gave a proof of the above inequality for $\SL_2(\mathbb{Z})\backslash \mathbb{H}$, and in \cite{GhoshReznikovSarnak2017nodal} they gave a proof for a compact surface of genus $g\geq0$. Since the proof is local, it is valid for non-compact surfaces in our case as mentioned in \cite{GhoshReznikovSarnak2013nodal,GhoshReznikovSarnak2017nodal}. See also \cite[\S6]{GhoshReznikovSarnak2017nodal} for a detailed discussion of the geometric conditions in the case $\Gamma_0(q)$.)
Note that by the Gauss--Bonnet formula we have $g\ll \vol(\Gamma_0(q)\backslash\mathbb{H}) = \frac{\pi}{3} (q+1)$ (see e.g. \cite[Eq. (2.17)]{iwaniec1997topics}). Hence we get
\[
  N^\beta \left(\phi_k\right) \gg_q S_\beta\left(\phi_k\right).
\]
%
Note that 
\[
  \|\phi_k\|^2_{L^2\left(\beta\right)} = \int_{\beta} |\phi_k(z)|^2 \dd s
  \ll \|\phi_k\|_\infty  \int_a^b |\phi_k(iy)| \frac{\dd y}{y}
  = \|\phi_k\|_\infty \|\phi_k\|_{L^1\left(\beta\right)} .
\]
Together with Theorem \ref{thm:sup-norm} and Lemmas \ref{lemma:rL2>>} and \ref{lemma:M<<}, for all but $O(K^{1-\varepsilon})$ forms in $\{\phi_k:K<k\leq 2K\}$, we have
\[
  N^\beta \left(\phi_k\right)
  \gg S_\beta\left(\phi_k\right)
  \gg \frac{\|\phi_k\|_{L^1\left(\beta\right)}}{M\left(\phi_k\right)\|\phi_k\|_2}
  \gg \frac{\|\phi_k\|^2_{L^2\left(\beta\right)}}{ M\left(\phi_k\right) \|\phi_k\|_2 \|\phi_k\|_\infty }
  \gg T^{1/8-\varepsilon}.
\]
This completes the proof of Theorem \ref{thm:nodal_domain}.



%


\end{document}